\documentclass[journal]{IEEEtran}

\usepackage{graphicx} 
\usepackage{amsmath, amssymb, amsfonts}
\usepackage{color}
\usepackage{subfig}
\usepackage{epstopdf}
\usepackage{setspace}

\newcommand{\eome}{{\rm \nabla\kern-.6em \nabla}}

\usepackage[bb=boondox]{mathalfa}

\newtheorem{theo}{Theorem}
\newtheorem{proposition}{Proposition}

\newtheorem{lemma}{Lemma}
\newtheorem{remark}{Remark}
\newenvironment{proof}[1][]{{\it Proof #1:~}}{\hfill$\square$\\}
\newcommand{\Chi}{\mathfrak{X}}

\newcommand{\supeq}{\geqslant}
\newcommand{\infeq}{\leqslant}
\newcommand{\V}{\mathcal{V}}

\newcommand{\R}{\mathbb{R}}
\renewcommand{\S}{\mathbb{S}}

\renewcommand\epsilon{\varepsilon}

\begin{document}

\title{Lyapunov stability analysis of a string equation coupled with an ordinary differential system}

\author{Matthieu~Barreau, Alexandre~Seuret, Fr\'ed\'eric~Gouaisbaut and Lucie~Baudouin
\thanks{M. Barreau, A. Seuret, F. Gouaisbaut and L. Baudouin are with LAAS - CNRS, Universit\'e de Toulouse, CNRS, UPS, France. This work is supported by the ANR project SCIDiS contract number 15-CE23-0014.\newline E-mail address: \textit{mbarreau,aseuret,fgouaisb,lbaudoui@laas.fr}.  }
}

\maketitle

\begin{abstract}
This paper considers the stability problem of a linear time invariant system in feedback with a string equation. A new Lyapunov functional is proposed using augmented states which enriches and encompasses the classical functionals of the literature. It results in tractable stability conditions expressed in terms of linear matrix inequalities. This methodology follows from the application of the Bessel inequality to the projections over the Legendre polynomials. Numerical examples illustrate the potential of our approach through three scenari: a stable ODE perturbed by the PDE, an unstable open-loop ODE and an unstable closed-loop ODE stabilized by the PDE.
\end{abstract}


\IEEEpeerreviewmaketitle
\vspace{-0.4cm}
\section{Introduction}

This paper presents a novel approach to assess stability of a heterogeneous system composed of the interconnection of a partial differential equation (PDE), more precisely a damped string equation, with a linear ordinary differential equation (ODE). While the topic of stability and control of PDE systems has a rich literature between applied mathematics \cite{coron2007control,lions1988exact} and automatic control \cite{luo2012stability}; the stability analysis (and the control) of such a coupled system belongs to a recent research area. To cite a few related results, one can refer to \cite{castillobuenaventura:hal-00718725, castillo2016dynamic, safi:hal-01354073} where an ODE is interconnected with a transport equation, to \cite{tang2011state} for a heat equation, \cite{krstic2009delay,cerpa:hal-01670643} for the wave equation and \cite{Wu20142787} for the beam equation.

Generally, the PDE is viewed as a perturbation to be compensated for instance using a backstepping method as proposed in \cite{krstic2011}, where infinite dimensional controllers are provided to cope with the undesirable effect of the PDE. Another interesting point of view relies on the converse approach: the ODE system can be seen as a finite dimensional boundary controller for the PDE (see \cite{andrea1994,morgul1995stabilization,Morgül2002731}). A last strategy describes a robust control approach, aiming at characterizing the robustness of the PDE-ODE interconnection \cite{espitia2016event}.



In the present paper, we consider a damped string equation, i.e. a stable one-dimensional wave equation which is connected at its boundary to a stable or unstable ODE. The proposed method to assess stability is inspired by the recent developments on the stability analysis of time-delay systems based on Bessel inequality and Legendre polynomials \cite{seuret:hal-01065142}. Since time-delay systems represent a particular class of systems coupling a transport PDE with a classical ODE  system (see for instance \cite{baudouin:hal-01310306}), the main motivation of this work is to show how this methodology can be adapted to a larger class of PDE/ODE systems as demonstrated with the heat equation in \cite{baudouinHeat}. 

Compared to the literature on coupled PDE/ODE systems, the proposed methodology aims at designing a new Lyapunov functional, integrating some cross-terms merging the ODE's and the PDE's usual terms. This new class of Lyapunov functional encompasses the classical notion of energy usually proposed in the literature by offering more flexibility. Hence, it allows to guarantee stability for a larger set of systems, for instance, unstable open-loop ODE and, for the first time to the best of our knowledge, even an unstable closed-loop ODE; meaning that the PDE helps for the stabilization. 

The paper is organized as follows. The next section formulates the problem and provides some general results on the existence of solutions and equilibrium. In Section~3, after a modeling phase inspired by the Riemann coordinates, a generic form of Lyapunov functional is introduced, and its associate analysis leads to a first stability theorem. Then, in Section~4, an extension using Bessel inequality is provided. Finally, Section~5 discusses the results on three examples. The last section draws some conclusion and perspectives.

\textbf{Notations:} $\Omega$ is the closed set $[0, 1]$ and $\mathbb{R}^+ = [0, +\infty)$. $(x,t) \mapsto u(x,t)$ is a multi-variable function from $\Omega \times \mathbb{R}^+$ to $\mathbb{R}$.
The notation $u_t$ stands for $\frac{\partial u}{\partial t}$. We also use the notations $L^2 = L^2(\Omega; \mathbb{R})$ and for the Sobolev spaces: ${H}^n = \{ z \in L^2; \forall m \infeq n, \frac{\partial^m z}{\partial x^m} \in L^2 \}$ and particularly $H^0 = L^2$. 
The norm in $L^2$ is $\|z\|^2 =  \int_{\Omega} |z(x)|^2  dx$. 
For any square matrices $A, B$, the operations `$\text{He}$' and `$\text{diag}$' are defined as follows: $\text{He}(A) = A + A^{\top}$ and $\text{diag}(A,B) = \left[ \begin{smallmatrix}A & 0\\ 0 & B \end{smallmatrix} \right]$.
A symmetric positive definite matrix $P$ of $\R^{n \times n}$ belongs to the set $\S^n_+$ or we write more simply $P \succ 0$.

\vspace{-0.4cm}
\section{Problem Statement}


We consider the coupled system described by
\begin{subequations}
	\begin{align}
		\dot{X}(t) &= AX(t) + Bu(1,t), & t \supeq 0, \label{eq:controller} \\
		u_{tt} (x,t) &= c^2 u_{xx}(x,t), & x \in \Omega, t \supeq 0, \label{eq:wave} \\
		u(0,t) &= K X(t), & t \supeq 0, \label{eq:boundary1} \\
		u_x(1,t) &= -c_0 u_t(1,t), & t \supeq 0, \label{eq:boundary2} \\
		u(x, 0) &=u^0(x), \ u_t(x, 0) = v^0(x), & x \in \Omega, \label{eq:initial1} \\
		X(0) & = X^0, & \label{eq:initial3}
	\end{align}
	\label{eq:problem}
\end{subequations}
\!\!with the initial conditions $X^0 \in \mathbb{R}^n$ and $(u^0, v^0) \in {H}^2 \times H^1$ such that equations \eqref{eq:boundary1} and \eqref{eq:boundary2} are respected. They are then called ``compatible'' with the boundary conditions. $A, B$ and $K$ are time-invariant matrices of appropriate size.

\begin{remark} When no confusion is possible, parameter $t$ may be omitted and so do the domains of definition. \end{remark}

This system can be viewed as an interconnection in feedback between a linear time invariant system \eqref{eq:controller} and an infinite dimensional system modeled by a string equation \eqref{eq:wave}. The latter is a one dimension hyperbolic PDE, representing the evolution of a wave of speed $c > 0$ and amplitude $u$. To keep the content clear, $u(x, t)$ is assumed to be a scalar but the calculations are done as if it was a vector of any dimension. The measurement we have access to is the state $u$ at $x = 1$ which is the right extremity of the string and the control is a Dirichlet actuation (equation \eqref{eq:boundary1}) because it affects directly the state $u$ and not its derivative. Another boundary condition must be added. It is defined at $x = 1$ by $u_x(1) = - c_0 u_t(1)$. This is a well-known damping condition when $c_0 > 0$ (see for example \cite{Lagnese1983163}). 
As presented in \cite{bresch2014output}, we find this kind of systems for instance when modeling a drilling mechanism. The control is then given at one end and the measurement is done at the other end.

More generally, this system can be seen either as the control of the PDE by a finite dimensional dynamic control law generated by an ODE \cite{chen1990wave} or on the contrary the robustness of a linear closed loop system with a control signal conveyed by a damped string equation. On the first scenario, both the ODE and the PDE are stable and the stability of the coupled system is studied. The second case corresponds to an unstable but stabilizable ODE connected to a stable PDE. 
To sum up, this paper focuses on the stability analysis of \textit{closed-loop} coupled system \eqref{eq:problem} with a potentially unstable closed-loop ODE but a stable PDE. This differs significantly from the backstepping methodology of \cite{krstic2009delay} which aims at designing an infinite dimensional control law ensuring the stability of a cascaded ODE-PDE system with a closed-loop stable ODE.


\vspace{-0.5cm}
\subsection{Existence and regularity of solutions} 

This subsection is dedicated to the existence and regularity of solutions $(X, u, u_t)$ to system \eqref{eq:problem}. We first introduce $\mathcal{H}^m = \mathbb{R}^n \times {H}^m \times H^{m-1}$ for $m \supeq 1$. We consider the classical norm on the Hilbert space $\mathcal{H} = \mathcal{H}^1 = \mathbb{R}^n \times {H}^1 \times L^2$:
\begin{equation*}
	\| (X, u, v) \|^2_{\mathcal H} = |X|_n^2 + \|u\|^2 + c^2\| u_x \|^2 + \| v \|^2.
\end{equation*}

This norm can be seen as the sum of the energy of the ODE system and the one of the PDE. 

\begin{remark}
	A more natural norm for space $\mathcal H$ would be $|X|_n^2 + \|u\|^2 + \| u_x \|^2 + \|v\|^2 $ which is equivalent to $\| \cdot \|^2_{\mathcal H}$. The norm used here makes the calculations easier in the sequel.
\end{remark}

Once the space is defined, we model system \eqref{eq:problem} using the following linear unbounded operator $T : \mathcal{D}(T) \to \mathcal H$:
\[
	T \left( \begin{smallmatrix} X \\ u \\ v \end{smallmatrix} \right) = \left( \begin{smallmatrix} A X + B u(1) \\ v \\ c^2 u_{xx} \end{smallmatrix} \right) \text{ and }
\]
\vspace{-0.2cm}
\begin{equation*}
	\mathcal{D}(T) = \left\{ (X,u,v) \in \mathcal H^2,  u(0) = K X, u_x(1) = -c_0 v(1) \right\}.
\end{equation*}

This operator $T$ is said to be dissipative with respect to a norm if its time-derivative along the trajectories generated by $T$ is strictly negative. The goal of this paper is then to find an equivalent norm to $\| \cdot \|_{\mathcal H}$ which allows us to refine the dissipativity analysis of $T$. This equivalent norm is derived from a general formulation of a Lyapunov functional, whose parameters  are chosen using a semi-definite programming optimization process.\\
Beforehand, from the semi-group theory, we propose the following result on the existence of solutions for \eqref{eq:problem}.

\begin{proposition} \label{sec:existence}
	If there exists a norm on $\mathcal{H}$ for which the linear operator $T$ is dissipative with $A+BK$ non singular, then there exists a unique solution $(X,u,u_t)$ of system \eqref{eq:problem} with initial conditions $(X^0, u^0, v^0) \in \mathcal{D}(T)$. Moreover, the solution has the following regularity property: $(X,u,u_t) \in C(0, +\infty, \mathcal H)$.
\end{proposition}

\begin{proof}This proof follows the same lines than in \cite{morgul1994dynamic}. Applying Lumer-Phillips theorem (p103 from \cite{tucsnak2009observation}), as the norm is dissipative, it is enough to show that for all $\lambda \in (0, \lambda_{max})$ with $\lambda_{max} > 0$, $\mathcal{D}(T) \subset \mathcal{R} \left( \lambda I - T \right)$ where $\mathcal{R}$ is the range operator. 
Let $(r, g, h) \in \mathcal{H}$, we want to show that for this system, there exists $(X, u, v) \in \mathcal{D}(T)$ for which the following set of equation is verified:
\begin{subequations} \label{eq:existence}
	\begin{align}
		\lambda X - AX - Bu(1) & = r, \label{eq:1}\\
		\lambda u(x) - v(x) & = g(x),  \label{eq:2}\\
		\lambda v(x) - c^2 u_{xx}(x) & = h(x),  \label{eq:3}
	\end{align}
\end{subequations}
for all $x \in \Omega$ and a given $\lambda > 0$. Equations \eqref{eq:2}, \eqref{eq:3} give:
\[
	\forall x \in (0, 1), \ \ u(x) = k_1 \exp(\lambda c^{-1} x) + k_2 \exp(-\lambda c^{-1} x) + G(x)
\]
where $G(x) = \int_0^x \frac{\lambda g(s) + h(s)}{\lambda c} \text{sinh}\left( \frac{\lambda}{c} (s - x) \right) ds \in H^2$. $k_1, k_2 \in \mathbb{R}$ are constants to be determined. Using the boundary condition $u(0) = KX$, we get:
\begin{equation*} \label{eq:u}
	\forall x \in (0, 1), \ \ u(x) = 2 k_1 \text{sinh}(\lambda c^{-1}x) + K X e^{-\lambda c^{-1}x} + G(x),
\end{equation*}
 Taking its derivative at the boundary we get:
\[
	u_x(1) = 2 \lambda c^{-1} k_1 \text{cosh}(\lambda c^{-1}) - \lambda c^{-1} KXe^{-\lambda c^{-1}} + G_1,
\]
with $G_1 \in \mathbb{R}$ known. We also have $u_x(1) + c_0v(1) = 0$, leading to $u(1) = G_2 + KX f(\lambda c^{-1})$ with $G_2 \in \mathbb{R}$ and $f(y) = \left(1 -  \tfrac{(c c_0 - 1)\text{sinh}(y) }{2(\text{cosh} (y) + c c_0 \text{sinh}(y))}\right) e^{-y}$. Then using  \eqref{eq:1}, we get: 
\[
	\left( \lambda I_n - (A+BK f(\lambda c^{-1})) \right)  X  = r +  BG_2.
\]

Since $f(\lambda c^{-1}) \to 1$ when $\lambda \to 0$ 
 and $A+BK$ is non singular, there exists $\lambda_{max} > 0$ such that $A + BKf(\lambda_{max} c^{-1})$ is non singular and 
\[
	\forall \lambda \in (0, \lambda_{max}), \quad \text{det} \left( \lambda I_n - (A+BK f(\lambda c^{-1})) \right) \neq 0.
\]

Then, there is a unique $X \in \mathbb{R}^n$ for a given $(r, f, h) \in \mathcal{H}$. We immediately get that $(X, u, v)$ is in $\mathcal{D}(T)$. Then for $\lambda \in (0, \lambda_{max}) $, $\mathcal{D}(T) \subset \mathcal{R}(\lambda I - T)$. The regularity property falls from Lumer-Phillips theorem.
\end{proof}
\vspace{-0.5cm}


\vspace{-0.35cm}
\subsection{Equilibrium point}
An equilibrium $x_{eq} = (X_e, u_e, v_e) \in \mathcal{D}(T)$ of system \eqref{eq:problem} is such that $Tx_{eq} =  (0_{n,1}, 0, 0) = 0_{\mathcal{H}}$, i.e. it verifies the following linear equations:
	\begin{subequations} \label{eq:equilibrium}
		\begin{align}
			0 & = AX_e+B u_e(1), \label{eq:equilibrium1}\\
			0 & = c^2 \partial_{xx} u_e(x), \quad \quad \quad x \in (0, 1), \label{eq:equilibrium2}\\
			v_e(x) & = 0, \quad \quad \quad \quad \quad \quad \quad \hspace{0.09cm} x \in (0,1),\\
			u_e(0) & = K X_e, \label{eq:equilibrium3}\\
			\partial_x u_e(1) & = 0. \label{eq:equilibrium4}
		\end{align}
	\end{subequations}
Using equation \eqref{eq:equilibrium2}, we get $u_e$ as a first order polynomial in $x$ but in accordance to equation \eqref{eq:equilibrium4}, $u_e$ is a constant function. Then, using equation \eqref{eq:equilibrium3}, we get $u_e = K X_e$. That leads to: $\left( A+BK \right)X_e = 0$. We obtain the following proposition:
	\begin{proposition} \label{sec:propEquilibrium}
		An equilibrium $(X_e, u_e, v_e) \in \mathcal{H}$ of system \eqref{eq:problem} verifies \mbox{$(A+BK) X_e = 0$}, $u_e=KX_e$, $v_e = 0$. Moreover, if $A+BK$ is not singular, system \eqref{eq:problem} admits a unique equilibrium $(X_e, u_e, v_e) = (0_{n,1}, 0, 0) = 0_{\mathcal{H}}$.
	\end{proposition}
\section{A First Stability Analysis Based on Modified Riemann Coordinates}
This part is dedicated to the construction of a Lyapunov functional. We introduce therefore a new structure based on variables directly related to the states of system \eqref{eq:problem}. 

\subsection{Modified Riemann coordinates}

The PDE considered in system \eqref{eq:problem} is of second order in time. As we want to use some tools already designed for first order systems, we propose to define some new states using modified Riemann coordinates, which satisfy a set of coupled first order PDEs and diagonalize the operator. Let us introduce these coordinates, defined as follows:
\vspace{-0.2cm}
\begin{equation*} \label{eq:sys2_3}
	\chi(x) =  \left[ \begin{matrix} u_t(x) + c u_x(x) \\ u_t(1-x) - c u_x(1-x) \end{matrix} \right] = \left[ \begin{matrix} \chi^{+}(x) \\ \chi^{-}(1-x) \end{matrix} \right].
\end{equation*}

The introduction of such variables is not new and the reader can refer to articles  \cite{prieur2016wave,bresch2014output} or \cite{4060979} and references therein about Riemann invariants. $\chi^+$ and $\chi^-$ are eigenfunctions of equation \eqref{eq:wave} associated respectively to the eigenvalues $c$ and $-c$. Therefore, using $\chi^{-}(1-x)$,  the previous equation leads to a transport PDE:

\begin{equation} \label{eq:sys2_2} 
	\forall t \geq 0, \forall x \in \Omega, \quad \chi_t(x,t) = c \chi_x(x,t).
\end{equation}

\begin{remark} The norm of the modified state $\chi$ can be directly related to the norm of the functions $u_t$ and $u_x$. Indeed simple calculations and a change of variables give:
\begin{equation}\label{eq:norm_chi_u}
	\|\chi\|^2 = 2 \left( \|u_t\|^2 + c^2 \|u_x\|^2 \right).
\end{equation}
\end{remark}
\begin{remark}This manipulation does not aim at providing an equivalent formulation for system \eqref{eq:problem} but at identifying a manner to build a Lyapunov functional for system \eqref{eq:problem}. \end{remark}



The second step is to understand how the extra-variable $\chi$ interacts with ODE \eqref{eq:controller}. Hence using \eqref{eq:boundary1}, we notice:
\begin{equation*}\begin{array}{lcl}
	\dot{X}& = &AX + B\left( u(1) - u(0) + KX \right),\\
	&= &(A+BK)X + B \int_0^1 u_x(x) dx.
\end{array}
\end{equation*}
To express the last integral term using $\chi$, we note that:
\begin{equation*}
	 2c\int_0^1 u_x(x) dx=\int_0^1 \chi^+(x)dx-\int_0^1\chi^-(x) dx.
\end{equation*}

This expression allows us to rewrite the ODE system as $\dot{X} = (A+BK)X + \tilde{B} \Chi_0$ where $\Chi_0 := \int_0^1 \chi(x) dx$ and $\tilde{B} = \frac{1}{2c} B \left[ \begin{smallmatrix} 1 & -1 \end{smallmatrix} \right]$.
The extra-state $\Chi_0$ follows the dynamics:
\begin{equation*}
	\dot{\Chi}_0  =  c \int_0^1 \chi_x(x) dx = c \left[ \chi(1) - \chi(0) \right].
\end{equation*}

The ODE dynamic can then be enriched by considering an extended system where $\Chi_0$ is viewed as a new dynamical state:
\begin{equation} \label{eq:extendedSystem}
	\begin{array}{rl}
		\dot{X}_0 = & \left[ \begin{smallmatrix} A+BK & \tilde{B} \\ 0_{2, n} & 0_{2} \end{smallmatrix} \right] X_0 + \left[ \begin{smallmatrix} 0_{n,2} \\ c I_2 \end{smallmatrix} \right] \left( \chi(1) - \chi(0) \right) , \\
	\end{array}
 \end{equation}
 with $ X_0 = [ X^{\top} \ \Chi^{\top}_0 ]^{\top}$. 
 Hence, associated to the original system \eqref{eq:problem}, we propose a set of equation \eqref{eq:sys2_2}-\eqref{eq:extendedSystem}. They are linked to system \eqref{eq:problem} but enriched by extra dynamics aiming at representing the interconnection between the extended finite dimensional system and the two transport equations. Nevertheless, these two systems are not equivalent. The transport equation gives trajectories of $u_t$ and $u_x$ but $u$ can be defined within a constant. The second set of equations just induces a formulation for a Lyapunov functional candidate which is developed in the subsection below.

\subsection{Lyapunov functional and stability analysis}

The main idea is to rely on the auxiliary variables satisfying equations \eqref{eq:sys2_2} and \eqref{eq:extendedSystem} to define a Lyapunov functional for the original system \eqref{eq:problem}. The associated Lyapunov function of ODE \eqref{eq:extendedSystem} is a simple quadratic term on the state $X_0^{\top} P_0 X_0$, with $P_0 \in \S^{n+2}_+$. It introduces automatically a cross-term between the ODE and the original PDE through $X_0$. Hence, the auxiliary equations of the previous paragraph shows a coupling between a finite dimensional LTI system and a transport PDE.
For the latter, inspired from the literature on time-delay systems \cite{baudouin:hal-01310306,4060979}, we provide a Lyapunov functional:
\begin{equation*}
	\V(u) = \int_0^1 \chi^{\top}(x) \left( S + xR \right) \chi(x) dx,
	\label{eq:Vcal}
\end{equation*}
with $S, R \in \mathbb{S}^2_+$.	 The use of the modified Riemann coordinates enables us to consider full matrices $S$ and $R$. 
As the transport described by the variable $\chi$ is going backward, $R$ is multiplied by $x$. 
Thereby, we propose a Lyapunov functional for system \eqref{eq:problem} expressed with the extended state variable $X_0$:
\begin{equation}
	V_0(X_0, u) = X_0^{\top} P_0 X_0 + \V(u).
	\label{eq:V1}
\end{equation}
This Lyapunov functional is actually made up of three terms:
\begin{itemize}
	\item A quadratic term in $X$ introduced by the ODE;
	\item A functional $\V$ for the stability of the string equation;
	\item A cross-term between $\Chi_0$ and $X$ described by the extended state $X_0$.
\end{itemize}
The idea is that this last contribution is interesting since we may consider the stability of system \eqref{eq:problem} with an unstable ODE, stabilized thanks to the string equation.
At this stage, a stability theorem can be derived using the Lyapunov functional $V_0$.

\begin{theo}\label{th0} Consider the system defined in \eqref{eq:problem} with a given speed $c$, a viscous damping $c_0 > 0$ with initial conditions $(X^0, u^0, v^0) \in \mathcal{D}(T)$. 
Assume there exist $P_0 \in \S^{n+2}_+$ and $S, R \in \mathbb{S}^2_+$ such that the following LMI holds: 
\begin{equation} 
	\label{eq:defPsi0}
	\!\Psi_0 \! = \! \text{He} \left( Z_0^{\top} P_0 F_0 \right) \! - \! c \tilde{R}_0 + c \! \left( H_0^{\top} (S+R) H_0 - G_0^{\top} S G_0 \right) \prec 0
\end{equation}
where
	\begin{equation} \label{eq:defTheo1}
		\!\!\begin{array}{ll}
			F_0 = \left[ \begin{matrix} I_{n+2} & \!0_{n+2, 2} \end{matrix} \right], & Z_0 = \left[ \begin{matrix} \mathcal{N}_0^{\top} & c(H_0 \! - \!G_0)^{\top} \end{matrix} \right]^{\!\top}\!\!, \\
			\mathcal{N}_0 = \left[ \begin{matrix} A\!+\!BK & \tilde{B} &  0_{n, 2} \end{matrix} \right] , &\tilde{R}_0 = \text{diag}\left( 0_n, R, 0_2 \right), \\
		\end{array}
	\end{equation}
	\[
		\!\!\begin{array}{ll}
			G_0 = \left[ \begin{matrix} 0_{2,n\!+\!2} &\!\!\!\! g \end{matrix} \right] \!+\! \left[ \begin{smallmatrix} K \\ 0_{1, n} \end{smallmatrix} \right] \mathcal{N}_0, \quad \quad & g = \left[ \begin{smallmatrix} 0 & 1 \\ 1 + c c_0 & 0 \end{smallmatrix} \right], \\
			H_0 = \left[ \begin{matrix} 0_{2, n\!+\!2} &\!\!\!\! h \end{matrix} \right] \!+\! \left[ \begin{smallmatrix} 0_{1, n} \\ K \end{smallmatrix} \right] \mathcal{N}_0, & h = \left[ \begin{smallmatrix} 1 - c c_0 & 0 \\ 0 & \!\!\!\!-1 \end{smallmatrix} \right].
		\end{array}
	\]
	
Then, there exists a unique solution to system \eqref{eq:problem} and it is exponentially stable in the sense of $\|\cdot \|_{\mathcal H}$ i.e. there exist $\gamma \supeq 1, \delta > 0$ such that the following estimate holds for $t > 0$:
	\begin{equation}
		\| (X(t), u(t), u_t(t)) \|^2_{\mathcal H} \infeq \gamma e^{-\delta t} \| (X^0, u^0, v^0) \|^2_{\mathcal H}.
		\label{eq:energyDecay}
	\end{equation}
\end{theo}

\begin{remark} \label{rem:S} 
The LMI  $\Psi_0\prec0$ includes a necessary condition given by $e_{3}^{\top} \Psi_0 e_{3} \prec 0,$ with $e_{3} = \left[ \begin{smallmatrix} 0_{2,n+2} & I_2 \end{smallmatrix} \right]^{\top}$, which is $h^\top (S+R)h-g^\top S g\prec 0$. This inequality is guaranteed if and only if the matrix $g^{-1}h$ has its eigenvalues inside the unit cycle of the complex plan, i.e. $c_0>0$, which is consistent with the result on exponential stability of \cite{helmicki1991ill}.  
\end{remark}

\subsection{Proof of Theorem \ref{th0}}

The proof of stability is presented below.
\subsubsection{Preliminaries}
As a first step of this proof, an inequality on $u$ is presented below.
\begin{lemma} \label{sec:normU}
	For $u \in H^1$, the following inequality holds:
	\[
		\|u\|^2 \infeq 2 \|u_x\|^2 + 2|u(0)|^2.
	\]
\end{lemma}
\vspace{-0.1cm}
\begin{proof}
Since $u_x \in H^1(\Omega)$, Young and Jensen inequalities imply that for all $x \in \Omega$:
	\begin{equation*} 
	u(x)^2 = \left( \int_0^x u_s(s) ds + u(0) \right)^2  \infeq 2 \int_0^x u_s^2(s) ds + 2 |u(0)|^2.
	\end{equation*}
	\vspace{-0.05cm}
\end{proof}
The proof of Theorem \ref{th0} consists in explaining how the LMI condition presented in the statement implies that there exist a functional $V$ and three positive scalars $\varepsilon_1, \varepsilon_2$ and $\varepsilon_3$ such that the following inequalities hold:
\vspace{-0.2cm}
\begin{equation} \label{eq:expo}
	\hspace*{-0.45cm}
	\begin{array}{llllllll}
		 \varepsilon_1 \| (X, u, u_t) \|^2_{\mathcal H} &\!\!\!\infeq &\!\!\!V(X, u) &\!\!\!\infeq &\!\quad\! \varepsilon_2 \| (X, u, u_t) \|^2_{\mathcal H}, \\
		&&\!\!\! \dot{V}(X, u) &\!\!\!\infeq &\!- \varepsilon_3 \| (X, u, u_t) \|^2_{\mathcal H}.
	\end{array}
\end{equation}

The next steps aim at proving \eqref{eq:expo} in order to obtain the convergence of the state to the equilibrium.

\subsubsection{Well-posedness} 
If the conditions of Theorem \ref{th0} are satisfied, then the inequality $\Psi_0 (1,1) = e_{1}^{\top} \Psi_0 e_{1} \prec 0$ holds where $e_{1} = \left[ \begin{smallmatrix} I_n & 0_{n,4} \end{smallmatrix} \right]^{\top}$. 
After some simplifications, we get $\text{He} \left( (A+BK)^{\top} Q \right) \prec 0$, for some matrix $Q$ depending on $R$, $S$ and $P_0$. This strict inequality requires that $A+BK$ is non singular and, in light of Propositions \ref{sec:existence} and \ref{sec:propEquilibrium}, the problem is indeed well-posed and $0_{\mathcal{H}}$ is the unique equilibrium point.
Furthermore, note that since $Q$ is not necessarily symmetric, then matrix $A+BK$ does not have to be Hurwitz.

\subsubsection{Existence of $\varepsilon_1$} 

	Conditions $P_0 \succ 0$ and $S, R \in \mathbb{S}^2_+$ mean that there exists $\varepsilon_1 > 0$, such that for all $x \in \Omega$:
	\vspace{-0.3cm}
	\[
		\begin{array}{rcl}
			P_0 \!\!&\succeq& \varepsilon_1 \text{diag} \left( I_{n} +2 K^{\top}K, 0_2 \right) , \\
			S+xR \ \succeq \ S &\succeq& \varepsilon_1 \frac{2 + c^2}{2 c^2} I_2.
		\end{array}
	\]
	These inequalities lead to:
	\[\begin{array}{lcl}
		V_0(X_0,u) \!\!\! & \supeq & \varepsilon_1 \left( |X|_n^2 +|KX|^2+  \frac{2 + c^2}{2 c^2} \|\chi\|^2 \right) \\
		&& + \int_0^1\chi^\top(x)\left(S+xR-\varepsilon_1\frac{2 + c^2}{2 c^2}I_2\right)\chi(x)dx\\
		&\supeq& \varepsilon_1 \left( |X|_n^2 +|KX|^2+  \frac{2 + c^2}{2 c^2} \|\chi\|^2 \right).\\
	\end{array}\]
	Using boundary condition \eqref{eq:boundary1} and equality \eqref{eq:norm_chi_u}, it becomes
	\[\begin{array}{lcl}
		V_0(X_0,u) \!\!\! &\supeq& \varepsilon_1 \left( |X|_n^2 +\|u\|^2+ \|u_t\|^2 + c^2\|u_x\|^2\right) \\
		&&+\frac{2\varepsilon_1}{c^2} \|u_t\|^2 + \varepsilon_1\left(2 \|u_x\|^2+2|u(0)|^2-\|u\|^2 \right).
	\end{array}\]
	
Then, we apply Lemma \ref{sec:normU} to ensure that the last term is positive. It follows that $V_0(X_0,u) \supeq \varepsilon_1 \ \| (X, u, u_t) \|^2_{\mathcal H}$, which ends the proof of existence of $\varepsilon_1$.

\subsubsection{Existence of $\varepsilon_2$}
	
Since $P_0 \in \mathbb{S}^{n+2}_+$ and $S, R \in \mathbb{S}^2_+$, there exists $\varepsilon_2 > 0$ such that for $x \in (0, 1)$:
\[
	\begin{array}{rcl}
		P_0 & \preceq & \text{diag} (\varepsilon_2 I_n,  \frac{\varepsilon_2}{4} I_2), \\
		S + xR \ \preceq \ S + R & \preceq & \frac{\varepsilon_2}{4} I_2 .
	\end{array}
\]

From equation \eqref{eq:V1}, we get:
\begin{equation}\label{ineqth0:eps2}
\begin{array}{lcl}
	\!\!\!V_0(X_0,u) \!\!\! &\infeq&  \varepsilon_2 \left(|X|_n^2 + \frac{1}{4} \Chi_0^{\top} \Chi_0 +\frac{1}{4} \|\chi\|^2 \right)\\
		&&+ \int_0^1 \chi^{\top}(x) \left(S+xR -\frac{\varepsilon_2}{4}I_2\right)\chi(x) dx \\
		&\infeq&  \varepsilon_2 \left(|X|_n^2 + \frac{1}{2} \|\chi\|^2\right)\\
\end{array}
\end{equation}
where we have used $\Chi_0^{\top} \Chi_0 \infeq  \|\chi\|^2$, which is a result of Jensen's inequality. The proof of the existence of $\varepsilon_2$ ends by using \eqref{eq:norm_chi_u} so that we get:
\begin{equation*}
	V_0(X_0,u) \infeq \!\varepsilon_2\!\left( |X|_n^2 \!+\!  \|u_t\|^2  \!+\!  c^2\|u_x\|^2 \right) \infeq \varepsilon_2 \| (X, u, u_t) \|^2_{\mathcal H}.
\end{equation*}

\subsubsection{Existence of $\varepsilon_3$}

Differentiating $V_0$ in \eqref{eq:V1} along the trajectories of system \eqref{eq:problem} leads to
\begin{equation*}
	\begin{array}{l}
		\dot{V}_0(X_0,u) = \text{He} \left( \left[ \begin{smallmatrix} \dot{X} \\ \dot{\Chi}_0 \end{smallmatrix} \right]^{\top} P_0 \left[ \begin{smallmatrix} X \\ \Chi_0 \end{smallmatrix} \right] \right) + \dot{\V}(u). \\
	\end{array}
\end{equation*}
Our goal is to express an upper bound of $\dot V_0$ thanks to the extended vector $\xi_0$ defined as follows:
	\begin{equation} \label{eq:xi1}
		\xi_0= \left[ \begin{matrix} X^{\top} & \Chi_0^{\top} & u_t(1) & c u_x(0) \end{matrix} \right]^{\top}.
	\end{equation}
	
Let us first concentrate on $\dot{\V}$.  Equation \eqref{eq:sys2_2} yields:
\begin{equation}
		\dot{\V}(u) = \displaystyle 2 c \int_0^1 \chi_x^{\top}(x,t) (S+xR) \chi(x,t) dx.
	\label{eq:Vdot2}
\end{equation}
Integrating by parts the last expression leads to:
\vspace{-0.3cm}
\begin{multline}
	\dot{\V}(u) =  c \left(  \chi^{\top}(1) (S+R) \chi(1) - \chi^{\top}(0) S \chi(0)  \vphantom{\int_0^1 \chi^{\top}(x) R \chi(x) dx } \right. \\
	\left. - \int_0^1 \chi^{\top}(x) R \chi(x) dx \right).
	\label{eq:Vcaldot}
\end{multline}

Then we note that $\dot X= \mathcal N_0\xi_0$, $\dot{\Chi}_0 = c(H_0 - G_0) \xi_0$, $\chi(1) = H_0 \xi_0$, $\chi(0) = G_0\xi_0,$ with $\xi_0$ defined in \eqref{eq:xi1} and the matrices above in \eqref{eq:defTheo1}. 
We get $X_0=F_0\xi_0$ and $\dot X_0 = Z_0\xi_0$ which results in the following expression for $\dot V_0$:
\begin{multline} \label{eq:Vdot}
	\!\!\!\!\!\! \dot{V}_0(X_0,u) \! =\xi_0^{\!\top} \!\! \left( \text{He} \left( \!Z_0^{\!\top} P_0 F_0 \right) \!+\! c H_0^{\!\top} \! ( S \!+\! R) H_0 \!-\! cG_0^{\!\top} S G_0 \right) \! \xi_0\\
 - c \int_0^1 \chi^{\top}(x) R \chi(x) dx.
\end{multline}
Then, using the definition of $\Psi_0$ given in \eqref{eq:defPsi0}, the previous expression can be rewritten as follows:
\begin{equation} \label{eq:Vdot13}
	\!\dot{V}_0(X_0,u) = \xi_0^{\!\top} \Psi_0 \xi_0 + c \Chi_0^{\top} R \Chi_0 - c \!\int_0^1\! \chi^{\top}(x) R \chi(x) dx.
\end{equation}
Since  $R \succ 0$ and $\Psi_0 \prec 0$, there exists $\varepsilon_3 > 0$ such that:
\begin{subequations}
	\begin{align}
		R \succeq & \ \frac{\varepsilon_3}{2c} \frac{2+c^2}{c^2} I_2, \label{eq:Rpos} \\
		\Psi_0 \preceq & - \varepsilon_3\text{diag} \left( I_n + 2 K^{\top}K, \frac{2+c^2}{2 c^2} I_2, 0_{2} \right) . \label{eq:PsiNeg}
	\end{align}
\end{subequations}
Using \eqref{eq:PsiNeg} and the boundary condition $u(0)=KX$, equation \eqref{eq:Vdot13} becomes:
\begin{equation*}
	\begin{array}{lcl}
		\dot{V}_0(X_0,u) \!\!\! & \! \infeq \! & - \varepsilon_3 \left( |X|^2_n + 2|u(0)|^2 + \frac{2+c^2}{2c^2} \|\chi\|^2 \right) \\
			&&+ c \Chi_0^{\!\top} \left(R - \frac{\varepsilon_3}{2c} \frac{2+c^2}{c^2}I_2 \right) \Chi_0 \\
			&&- c \int_0^1 \chi^{\!\top}(x) \left(R - \frac{\varepsilon_3}{2c} \frac{2+c^2}{c^2}I_2 \right) \chi(x) dx,
	\end{array}
\end{equation*}
so that we get by application of Jensen's inequality:
\begin{equation} \label{ineqth0:eps3}
	\dot{V}_0(X_0,u)  \infeq - \varepsilon_3 \left( |X|^2_n + 2|u(0)|^2 + \frac{2+c^2}{2c^2} \|\chi\|^2 \right).
\end{equation} 

The most important part of the proof lies in the following trick. Since \eqref{eq:norm_chi_u} holds, we get:
\begin{equation*}
	\begin{array}{lcl}
		\dot{V}_0(X_0,u)\!\!\!\!& \infeq \!\!\! &
		- \varepsilon_3\|(X,u, u_t)\|_{\mathcal H}^2  -\varepsilon_3\frac{2}{c^2}\|u_t\|^2\\
		&&- \varepsilon_3 \left( 2|u(0)|^2+ 2\|u_x\|^2-\|u\|^2\right).
	\end{array}
\end{equation*}
\\
Moreover, Lemma \ref{sec:normU} ensures that the last term of the previous expression is negative so that we have
$\dot{V}_0(X_0,u) \infeq - \varepsilon_3 \| (X, u, u_t) \|^2_{\mathcal H}$, which concludes this proof of existence.
		
\subsubsection{Conclusion} Finally, there exist $\varepsilon_1, \varepsilon_2, \varepsilon_3 \!\!>\!\! 0$ such that \eqref{eq:expo} holds for a functional $V_0$. Hence $V_0(\cdot)$ defines an equivalent norm to $\|\cdot\|_{\mathcal{H}}$ and is dissipative. It means, according to Propositions \ref{sec:existence} and \ref{sec:propEquilibrium}, that there exists a unique solution to system \eqref{eq:problem} in $\mathcal{H}$.
	Equation \eqref{eq:expo} also brings: $\dot{V}_0(X_0,u) + \frac{\varepsilon_3}{\varepsilon_2} V_0(X_0,u) \infeq 0$ 
	and
	\begin{equation*}
\forall t > 0, \quad \| (X(t), u(t), u_t(t)) \|^2_{\mathcal H} \infeq \frac{\varepsilon_2}{\varepsilon_1}e^{-\frac{\varepsilon_3}{\varepsilon_2} t} \| (X^0, u^0 ,v^0) \|^2_{\mathcal H} ,
	\end{equation*}
which shows the exponential convergence of all the trajectories of system \eqref{eq:problem} to the unique equilibrium $0_{\mathcal{H}}$. In other words, the solution to system \eqref{eq:problem} is exponentially stable.

\begin{remark}
It is also worth noting that LMI \eqref{eq:defPsi0} can be transformed to extend this theorem to uncertain ODE systems subject to polytopic-type uncertainties for instance.
\end{remark}

\section{Extended Stability Analysis}
In the previous analysis, we have proposed an auxiliary system presented in \eqref{eq:sys2_2}-\eqref{eq:extendedSystem} helping us to define a new Lyapunov functional for system~\eqref{eq:problem}. The notable aspect is that the term $\Chi_0 = \int_0^1 \chi(x) dx$ appears naturally in the dynamics of system~\eqref{eq:problem}. In light of the previous work on integral inequalities in \cite{seuret:hal-01065142}, this term can also be interpreted as the projection of the modified state $\chi$ over the set of constant functions in the sense of the canonical inner product in $L^2$. One may therefore enrich \eqref{eq:extendedSystem} by additional projections of $\chi$ over the higher order Legendre polynomials, as one can read in \cite{seuret:hal-01065142,baudouin:hal-01310306} in the context of time-delay systems. The family of shifted Legendre polynomials, denoted $\left\{ \mathcal{L}_k \right\}_{k \in \mathbb{N}}$ and defined over  $[0, 1]$  by $\mathcal{L}_k(x) = (-1)^k \sum_{l = 0}^k (-1)^l \left( \begin{smallmatrix} k \\ l \end{smallmatrix} \right) \left( \begin{smallmatrix} k+l \\ l \end{smallmatrix} \right) x^l$ with $\left( \begin{smallmatrix} k \\ l \end{smallmatrix} \right) = \frac{k!}{l! (k-l)!}$, form an orthogonal family with respect to the $L^2$ inner product (see \cite{courant1966courant} for more details). 

\subsection{Preliminaries}
The previous discussion leads to the definition of the projection of any function $\chi$ in $L^2$ on the family $\left\{ \mathcal{L}_k \right\}_{k \in \mathbb{N}}$:
\[
	\forall k \in \mathbb N, \quad \Chi_k := \int_0^1 \chi(x) \mathcal{L}_k(x) dx, 
\]
An augmented vector $X_N$ is naturally derived for any $N \in \mathbb N$:
 \begin{equation} \label{eq:notationN}
 	\begin{array}{l}
 		X_N = \left[ \begin{matrix} X^{\top} & \Chi^{\top}_0 & \cdots & \Chi^{\top}_{N} \end{matrix} \right]^{\top}.
	\end{array}
\end{equation}

Following the same methodology as in Theorem \ref{th0}, this specific structure suggests to introduce a new Lyapunov functional, inspired from \eqref{eq:V1}, with $P_N \in \mathbb{S}^{n+2(N+1)}_+$:
\begin{equation} \label{eq:VN}
	V_N(X_N, u) = X_N^{\top} P_N X_N + \V(u).
\end{equation}

In order to follow the same procedure, several technical extensions are required. Indeed, the stability conditions issued from the functional $V_0$ are proved using Jensen's inequality and an explicit expression of the time derivative of $\Chi_0$. Therefore, it is necessary to provide an extended version of Jensen's inequality and of this differentiation rule. These technicals steps are summarized in the two following  lemmas. 

\begin{lemma}\label{Bess}
	For any function $\chi \in L^2$ and symmetric positive matrix $R \in \mathbb S^2_+$, the following Bessel-like integral inequality holds for all $N\in \mathbb N$:
	\begin{equation}\label{eq:Bessel}
		\int_{0}^1 \chi^{\top}(x) R \chi(x) dx  \supeq \sum_{k=0}^{N} (2k+1)  \Chi_k^{\top} R \Chi_k.
	\end{equation}
\end{lemma}

This inequality includes Jensen's inequality as the particular case $N=0$, suggesting that this lemma is an appropriate extension and should help to address the stability analysis using the new Lyapunov functional \eqref{eq:VN} with the augmented state $X_N$ defined in \eqref{eq:notationN}. 

The proof of Lemma~\ref{Bess} is based on the expansion of the positive scalar $\| R^{1/2}\chi_N\|^2$ 
where $\chi_N(x)=\chi(x)- \sum_{k=0}^{N} (2k+1) \Chi_k L_k(x)$ can be interpreted as the approximation error between $\chi$ and its orthogonal projection over the family $\{ \mathcal{L}_k \}_{k \infeq N}$. 

The next lemma is concerned by the differentiation of $\Chi_k$.
\begin{lemma}\label{lem:Chi_k}
	For any function $\chi \in L^2$, the following expression holds for any $N$ in $\mathbb N$:
	\begin{equation*}
		\left[ \begin{smallmatrix}  \dot \Chi_0 \\ \vdots \\ \dot \Chi_{N} \end{smallmatrix} \right] = c\mathbb 1_N\chi(1)-c\bar {\mathbb 1}_N\chi(0) -cL_N \left[ \begin{smallmatrix}  \Chi_0 \\ \vdots \\ \Chi_{N} \end{smallmatrix} \right],
	\end{equation*}
	where
	 \vspace{-0.05cm}
	\begin{equation}\label{def_LN}
		\begin{array}{rcllcl}
			\!\!\! L_N\!=\!\left[\!\begin{smallmatrix} 
			\ell_{0,0}I_2 & \cdots& 0_2 \\ 
			\vdots & \ddots &\vdots\\ 
			\ell_{N,0}I_2 & \cdots & \ell_{N,N}I_2 \\ 
			\end{smallmatrix}\!\right]\!\!,
			\mathbb{1}_N\!=\!\left[\!\begin{smallmatrix} I_2  \\ \vdots \\ I_2\end{smallmatrix}\!\right]\!\!,
			\bar{\mathbb 1}_N\!=\!\left[\!\begin{smallmatrix}  I_2  \\  \vdots \\(-1)^{N} I_2\end{smallmatrix}\!\right],
		\end{array}
	\end{equation}
	with $\ell_{k,j} = (2j+1)(1 - (-1)^{j+k})$ if $ j \infeq k$ and $0$ otherwise.
\end{lemma}

\begin{proof} The proof of this lemma is presented in appendix because of its technical nature. \end{proof}
\vspace{-0.7cm}
\subsection{Main result}

Taking advantage of the previous lemmas, the following extension to Theorem \ref{th0} is stated:
\begin{theo} \label{sec:theoN}
	Consider system \eqref{eq:problem} with a given speed $c > 0$, a viscous damping $c_0~>~0$ and initial conditions $(X^0, u^0, v^0) \in \mathcal D(T)$. 
	 Assume that, for a given integer $N \in \mathbb N$, there exist $P_N~\in~\S^{n+2(N+1)}_+$ and $S, R \in \mathbb{S}_+^2$ such that inequality 
	 \vspace{-0.1cm}
	\begin{multline} \label{eq:psiN}
		\Psi_N = \text{He} \left( Z_N^{\top} P_N F_N \right) - c\tilde{R}_N\\
		 + c \left( H_{N}^{\top} (S+R) H_{N} - G_{N}^{\top} S G_{N}\right)\prec 0
	\end{multline}
	 \vspace{-0.05cm}
	\!\!holds, where
	\begin{equation} \label{eq:defTheo2}
		\begin{array}{ll}
			\!\!\!\!\!\! F_N \!\!\!\! &= \left[ \begin{matrix} I_{n+2(N+1)} & 0_{n+2(N+1), 2} \end{matrix} \right], Z_N = \left[ \begin{matrix} \mathcal{N}_N^{\top} & c\mathcal Z_{N}^{\top} \end{matrix} \right]^{\!\top}\!\!\!, \\  
			\!\!\!\!\!\! \mathcal{N}_N \!\!\!\! &= \left[ \begin{matrix} A+BK&\tilde B & 0_{n, 2(N+1)} \end{matrix} \right], \\
			\!\!\!\!\!\! \mathcal Z_{N} \!\!\!\! &= \mathbb 1_N H_{N}\! -\! \bar{\mathbb 1}_N G_{N}\! -\! \left[\begin{matrix}  0_{2N\!+\!2,n} &\!\!\!\! L_{N} &\!\!\!\! 0_{2N\!+\!2,2}\end{matrix} \right], \\
			\!\!\!\!\!\! G_{N} \!\!\!\! &= \left[ \begin{matrix} 0_{2,n+2(N+1)} & g \end{matrix} \right] + \left[ \begin{smallmatrix} K \\ 0_{1, n} \end{smallmatrix} \right] \mathcal{N}_N, \\
			\!\!\!\!\!\! H_{N} \!\!\!\! &= \left[ \begin{matrix} 0_{2, n+2(N+1)} & h \end{matrix} \right] + \left[ \begin{smallmatrix} 0_{1, n} \\ K \end{smallmatrix} \right] \mathcal{N}_N, \\
			\!\!\!\!\!\! \tilde{R}_N \!\!\!\! &= \text{diag}\left( 0_n, R,3R, \cdots, (2N+1) R , 0_2 \right),
		\end{array}
	\end{equation}
and where matrices $L_N$, $\mathbb {1}_N$ and $\bar{\mathbb {1}}_N$ are given in \eqref{def_LN}. 

Then, the coupled infinite dimensional system \eqref{eq:problem} is exponentially stable in the sense of norm $\|\cdot \|^2_{\mathcal H}$ and there exist $\gamma>1$ and $\delta>0$ such that energy estimate \eqref{eq:energyDecay} holds.
\end{theo}

\begin{remark} Remark \ref{rem:S} also applies for this theorem and it means that $c_0$ must be strictly positive. In other words, these theorems cannot ensure the stability of the interconnection if the PDE is undamped. \\
Also note that Theorem \ref{sec:theoN} with $N = 0$ leads exactly to the same conditions as presented in Theorem \ref{th0}.
\end{remark}
\begin{remark}
This methodology introduces a hierarchy in the stability conditions inspired from what one can read in \cite{seuret:hal-01065142} in the case of time-delay systems. More precisely, the sets
\[
	\mathcal{C}_{N}=\left\{c>0\mbox{ s.t. } \exists P_N\in \mathbb S^{n+2(N+1)}_+,S,R\in \mathbb S^{2}_+, \Psi_N\prec0 \right\}
\]
representing the parameters $c$ for which the LMI of Theorem~2 is feasible for a given system \eqref{eq:problem} and for a given $N\in \mathbb N$, satisfy the following inclusion: $\mathcal{C}_N \subseteq \mathcal{C}_{N+1}$. In other words, if there exists a solution to Theorem \ref{sec:theoN} at an order $N_0$, then there also exists a solution at any order $N\geq N_0$. The proof is very similar to the one given in \cite{seuret:hal-01065142}. We can proceed by induction with $P_{N+1} = \left[ \begin{smallmatrix} P_N & 0 \\ 0 & \varepsilon I_2 \end{smallmatrix} \right]$ and a sufficiently small $\varepsilon > 0$. Then, $\Psi_{N} \prec 0 \Rightarrow \Psi_{N+1} \prec 0$. The calculations are tedious and technical and we do not intend to give them in this article. 
\end{remark}
\vspace{-0.3cm}
\subsection{Proof of Theorem 2}
\vspace{-0.1cm}
	The proof of dissipativity follows the same line as in Theorem~\ref{th0} and consists in proving the existence of positive scalars $\varepsilon_1, \varepsilon_2$ and $\varepsilon_3$ such that the functional $V_N$ verifies the inequalities given in \eqref{eq:expo}. 
\subsubsection{Well-posedness} Using a similar reasoning to Theorem~\ref{th0}, a necessary condition for LMI \eqref{eq:psiN} to be verified is that $A+BK$ is non singular. Then, according to Propositions~\ref{sec:existence}~and~\ref{sec:propEquilibrium}, the problem is well-posed and $0_{\mathcal{H}}$ is the unique equilibrium.

\subsubsection{Existence of $\varepsilon_1$} 

	It strictly follows the proof in Theorem~\ref{th0} and is therefore omitted.

\subsubsection{Existence of $\varepsilon_2$} 
	Since $P_N, S$ and $R$ are definite positive matrices, there exists $\varepsilon_2 > 0$ such that:
	\vspace{-0.1cm}
	\[
		\begin{array}{rcl}
			P_N & \preceq & \text{diag} \left( \varepsilon_2 I_n, \frac{\varepsilon_2}{4} \text{diag} \left\{ (2k+1) I_n \right\}_{k \in (0, N)} \right), \\
			(S+xR) & \preceq & S + R \ \preceq \ \frac{\varepsilon_2}{4} I_2,\quad \forall x\in (0,1).
		\end{array}
	\]
	Then, from equation \eqref{eq:VN}, we get:
	\vspace{-0.1cm}
	\begin{equation*}
		\begin{array}{lll}
			V_N(X_N, u) 
				&\!\!\!\! \infeq &\!\!\!\! \displaystyle\varepsilon_2 |X|_n^2  \vphantom{\sum_{k=0}^{N}} \! +\! \frac{\varepsilon_2}{4}\! \left( \sum_{k=0}^{N} (2k\!+\!1) \Chi_k^{\top} \Chi_k + \| \chi \|^2 \right) \\
&\!\!\!\! \infeq &\!\!\!\! \varepsilon_2 \left( |X|_n^2\! +\! \frac{1}{2} \|\chi\|^2 \right).
		\end{array}
	\end{equation*}
	
While the first inequality is guaranteed by the constraint $(S+xR) \preceq \frac{\varepsilon_2}{4} I_2$, for all $x\in (0,1),$ the second estimate results from the application of Bessel inequality \eqref{eq:Bessel}. Therefore, following the same procedure as in the proof of Theorem~\ref{th0} after equation \eqref{ineqth0:eps2}, there indeed exists $\varepsilon_2 > 0$ such that $V_N(X_N,u) \infeq \varepsilon_2  \| (X, u, u_t) \|^2_{\mathcal H}$.

\subsubsection{Existence of $\varepsilon_3$} 

	Differentiating in time $V_N$ defined in \eqref{eq:VN} along the trajectories of system \eqref{eq:problem} leads to:
	\begin{equation*}
		\dot{V}_N(X_N, u) = \text{He} \left( \left[ \begin{smallmatrix} \dot{X} \\  \dot{\Chi}_0 \\ \vdots \\ \dot{\Chi}_N \end{smallmatrix} \right]^{\top} P_N \left[ \begin{smallmatrix} {X} \\  {\Chi}_0 \\ \vdots \\ {\Chi}_N\end{smallmatrix} \right]  \right) + \dot{\V}(u).
	\end{equation*}

	The goal here is to find an upper bound of $\dot{V}_N$ using the following extended state: $\xi_N = \left[ \begin{matrix} X_N^{\top} & u_t(1) & c u_x(0) \end{matrix} \right]^{\top}$.
	Using equation \eqref{eq:Vcaldot} and Lemma \ref{lem:Chi_k}, we note that $X_N = F_N\xi_N, ~ \dot X_N = Z_N\xi_N, ~ \chi(1) = H_{N}\xi_N, ~ \chi(0) = G_{N}\xi_N$ where matrices $F_N, Z_N, H_{N}, G_{N}$ are given in \eqref{eq:defTheo2}. Then we can write:
	\vspace{-0.5cm}
	\begin{multline}
		\dot{V}_N(X_N, u) = \xi_N^{\top} \Psi_N \xi_N + c \sum_{k=0}^{N} \Chi_k^{\top} (2k+1) R \Chi_k  \\
		-c  \int_0^1 \chi^{\top}(x) R \chi(x) dx.
		\label{eq:Vdot23}
	\end{multline}
	
Since  $R \succ 0$ and $\Psi_N \prec 0$, there exists $\varepsilon_3 > 0$ such that:
	\begin{equation} \label{eq:Rpos2} 
		\begin{array}{lcl}
			R &\succeq & \ \frac{\varepsilon_3}{2c} \frac{2+c^2}{c^2} I_2,\\
			\Psi_N &\preceq & 
							-\varepsilon_3 \text{diag} \left(  I_n + K^{\top} K, \right. \\
&&\ \left. \frac{2+c^2}{2c^2} \text{diag} \{ I_2,3I_2,\dots,(2N\!+\!1) I_2 \}, 0_2 \right)\!.
		\end{array}
	\end{equation}
Using \eqref{eq:Rpos2} and Bessel's inequality, equation \eqref{eq:Vdot23} becomes:
	\begin{multline*}
		\dot{V}_N(X_N,u) \infeq - \varepsilon_3 \left( |X|^2_n + 2|u(0)|^2 + \frac{2+c^2}{2c^2} \| \chi \|^2 \right),
	\end{multline*}
which is comparable to equation \eqref{ineqth0:eps3}. Therefore, similarly, we obtain $\dot{V}_N(X_N,u) \infeq - \varepsilon_3 \ \| (X, u, u_t) \|^2_{\mathcal H}$.

\subsubsection{Conclusion}
There exist $\varepsilon_1, \varepsilon_2$ and $\varepsilon_3$ positive reals such that inequalities \eqref{eq:expo} are satisfied and the exponential stability of system \eqref{eq:problem} is therefore guaranteed.

\vspace{-0.3cm}
\section{Examples}

Three examples of stability for system \eqref{eq:problem} are provided here. In each case, $A+BK$ is non singular and therefore, there is a unique equilibrium. The solver used for the LMIs is ``sedumi'' with the YALMIP toolbox \cite{1393890}. The dashed curve denoted ``Freq'' is obtained using a frequency analysis and displays the exact stability area. This exact method is explained in \cite{barreauInputOutput} and does not use Lyapunov arguments.
\vspace{-0.5cm}
\subsection{Problem \eqref{eq:problem} with $A$ and $A+BK$ Hurwitz}
In this first numerical example, the considered system is defined as follows:
\vspace{-0.1cm}
\begin{equation} \label{eq:simu1}
	\begin{array}{ccc}
		A = \left[ \begin{smallmatrix} -2 & 1 \\ 0 & -1 \end{smallmatrix} \right], &
		B = \left[ \begin{smallmatrix} 1 \\ 1 \end{smallmatrix} \right], &
		K = \left[ \begin{smallmatrix} 0 & -2 \end{smallmatrix} \right].
	\end{array}
\end{equation}
Matrices $A$ and $A+BK$ are Hurwitz. The ODE and the PDE are then stable if they are not coupled. As shown in Figure \ref{fig:simu1}, the frequency argument shows that there exists a minimum wave speed called here $c_{min}$ which is function of the damping $c_0$ for system \eqref{eq:problem} to be stable.

\begin{figure}
	\centering
	\subfloat[System \eqref{eq:simu1}: $A$ and $A+BK$ Hurwitz]{\includegraphics[width=8.3cm]{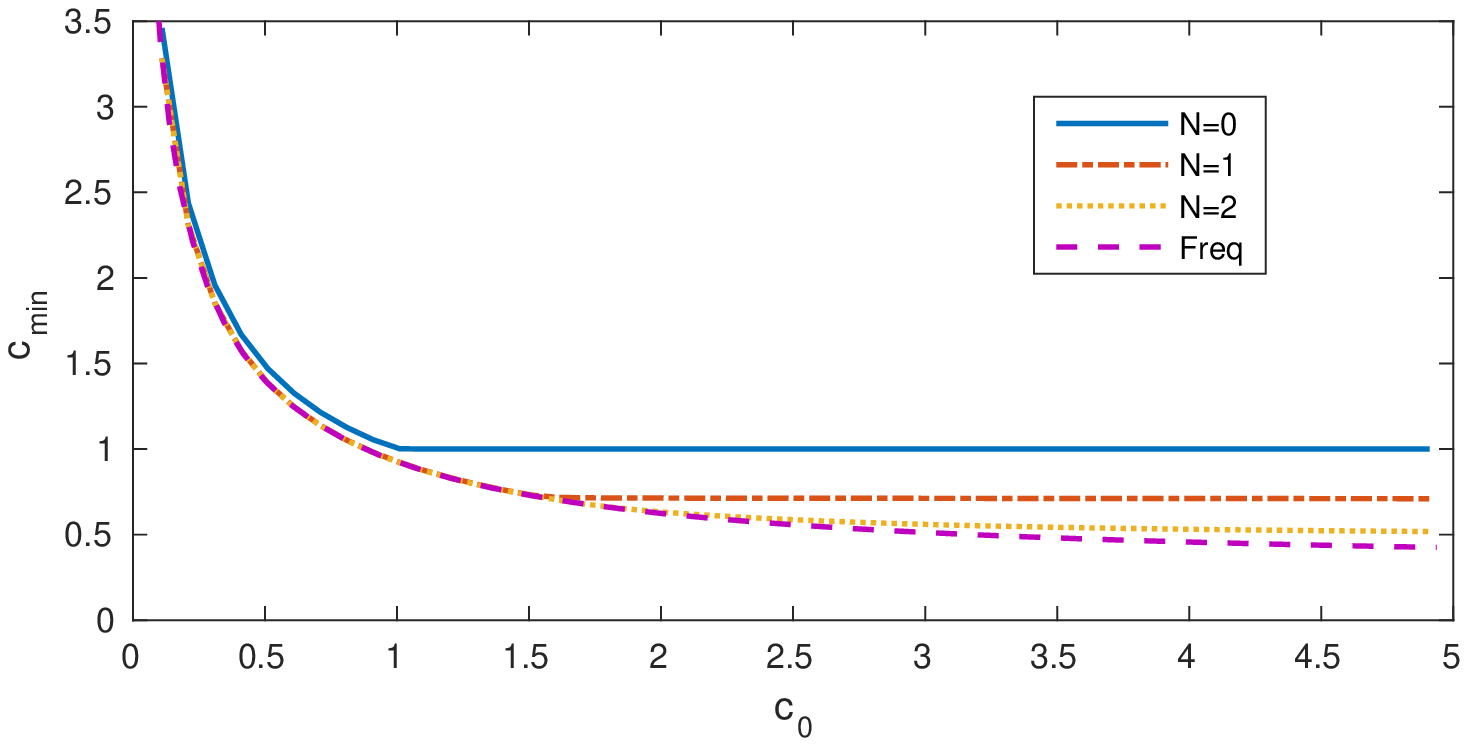} \label{fig:simu1}}\\
	\vspace{-0.4cm}
	\subfloat[System \eqref{eq:simu2}: $A$ not Hurwitz and $A+BK$ Hurwitz]{\includegraphics[width=8.3cm]{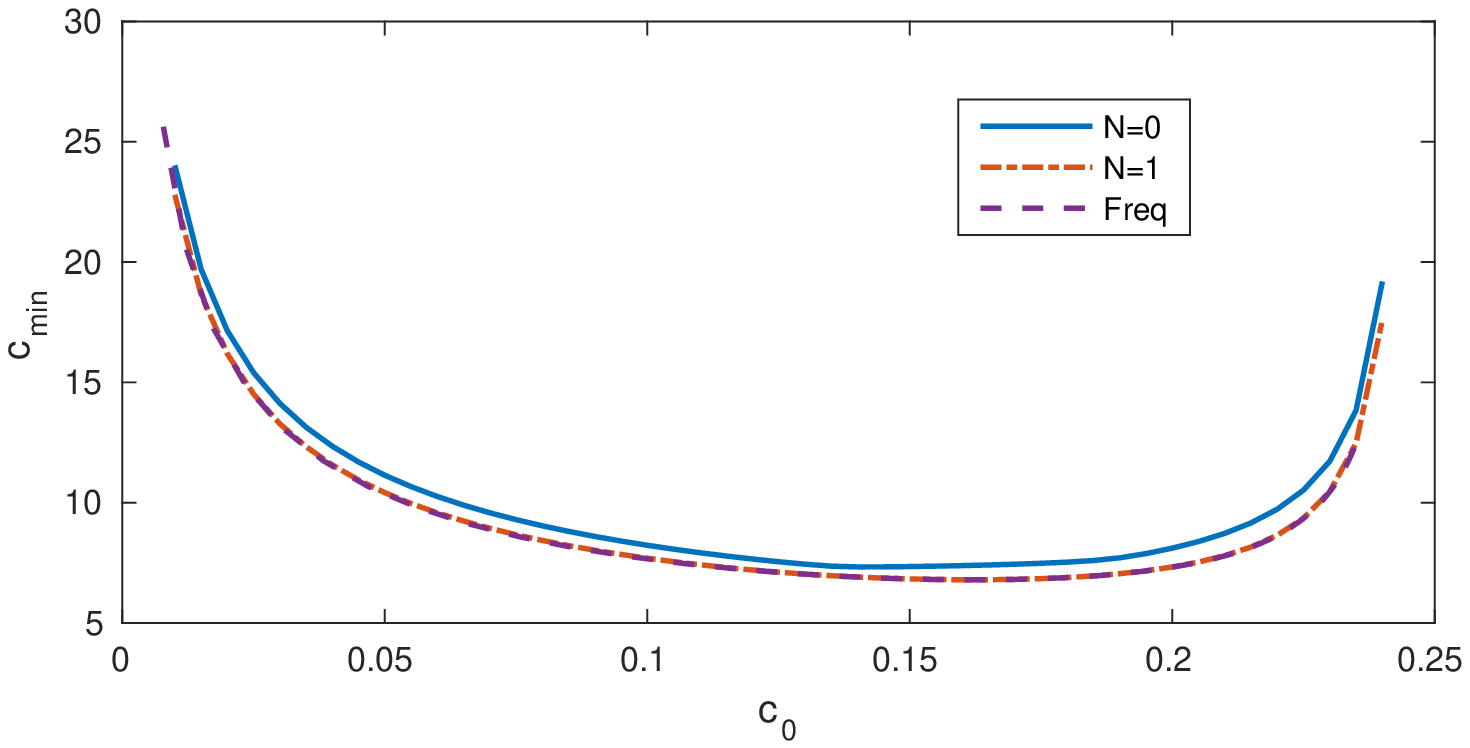} \label{fig:simu2}}\\
	\vspace{-0.4cm}
	\subfloat[System \eqref{eq:simu3}: $A$ and $A+BK$ not Hurwitz]{\includegraphics[width=8.3cm]{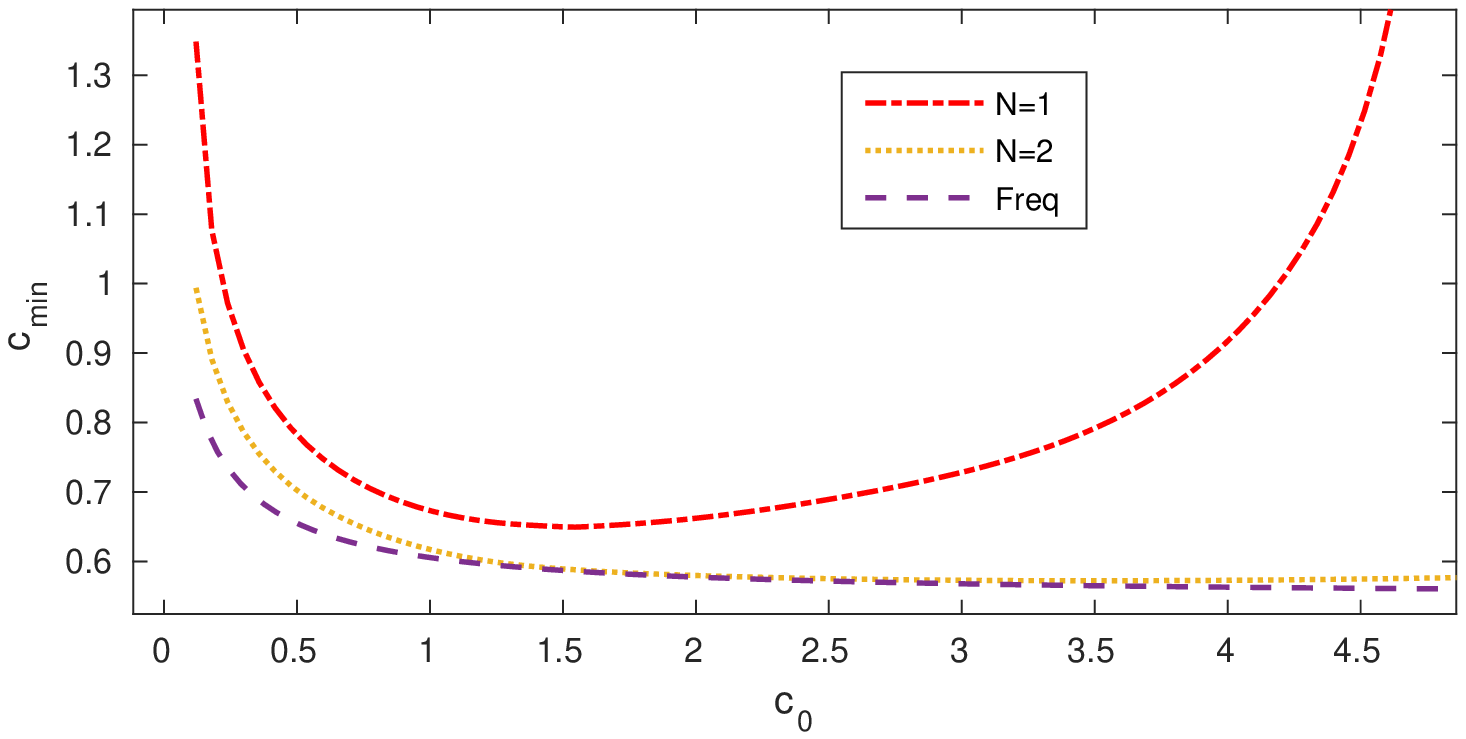} \label{fig:simu3}}
	\vspace{-0.1cm}
	\caption{Minimum wave speed $c_{min}$ as a function of $c_0$ for system \eqref{eq:problem} to be stable. The values for $A$, $B$ and $K$ are given by equations \eqref{eq:simu1}, \eqref{eq:simu2} or \eqref{eq:simu3}.}
\end{figure}

The phenomenon induced by the coupling can be understood as the robustness of the ODE to a disturbance generated by a wave equation. 
Intuitively, if the wave speed is large enough, the perturbation tends to $0$ fast enough for the ODE to keep its stability behavior. Another interesting thing to notice is the decrease of $c_{min}$ as $N$ increases (hierarchy of the stability criteria with respect to the order $N$). For this example, as $N$ increases, the stability area is converging to the exact one.
\vspace{-0.4cm}
\subsection{Problem \eqref{eq:problem} with $A+BK$ Hurwitz and $A$ not Hurwitz.}

Let us consider here, system \eqref{eq:problem} described by the following matrices:
\vspace{-0.3cm}
\begin{equation} \label{eq:simu2}
	\begin{array}{ccc}
		A = \left[ \begin{smallmatrix} 2 & 1 \\ 0 & 1 \end{smallmatrix} \right], &
		B = \left[ \begin{smallmatrix} 1 \\ 1 \end{smallmatrix} \right], &
		K = \left[ \begin{smallmatrix} -10 & 2 \end{smallmatrix} \right].
	\end{array}
\end{equation}

As $A$ is not Hurwitz, we are studying the stabilization of the ODE through a communication medium modeled by the wave equation. For the same reason as before, the wave speed must be large enough for the control to be not too much delayed but also with a moderated damping to transfer the state variable $X$ through the PDE equation. Then, a $c_{0max}$ is appears as one can see in Figure \ref{fig:simu2}.


Some numerical simulations have been performed on this example. 
Figure \ref{fig:simu2} shows that for system \eqref{eq:simu2}  with $c_0 = 0.15$, the minimum wave speed is $c_{min} = 6.83$. The numerical stability can also be seen in Figure \ref{fig:3sim} and indeed, the system is at the boundary of the stable area in Figure \ref{fig:3sim2} and unstable for smaller values of $c$. The results coming from the exact criterion and Theorem 2 are close even for small $N$. That means the stability area provided with $N = 1$ is a good estimation of the maximum stability set.


\begin{figure*}
	\centering
	\subfloat[$c = 10$]{\includegraphics[width=6cm]{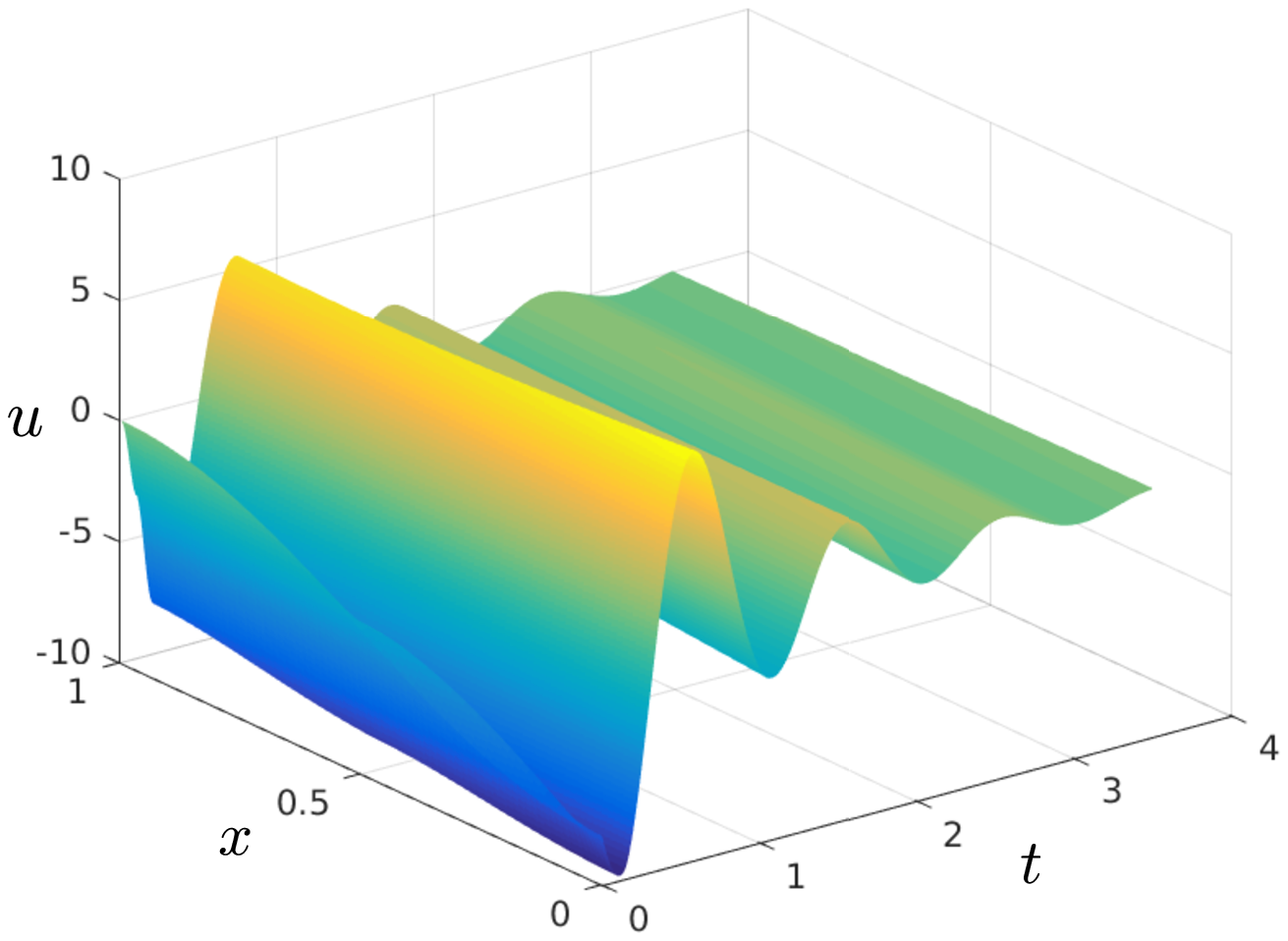}\label{fig:3sim1}}
	\subfloat[$c = 6.83$]{\includegraphics[width=6cm]{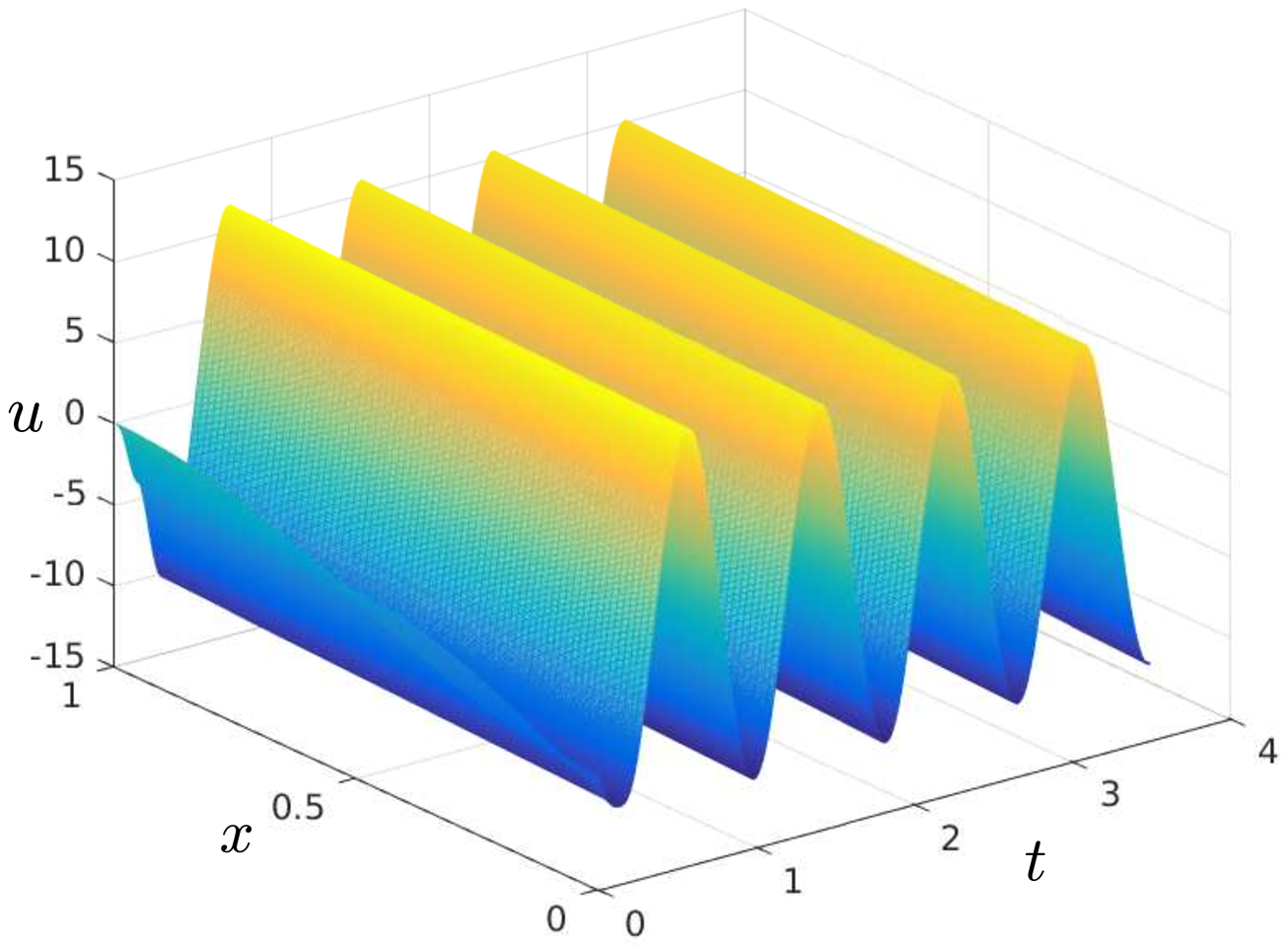}\label{fig:3sim2}}
	\subfloat[$c = 6.5$]{\includegraphics[width=6cm]{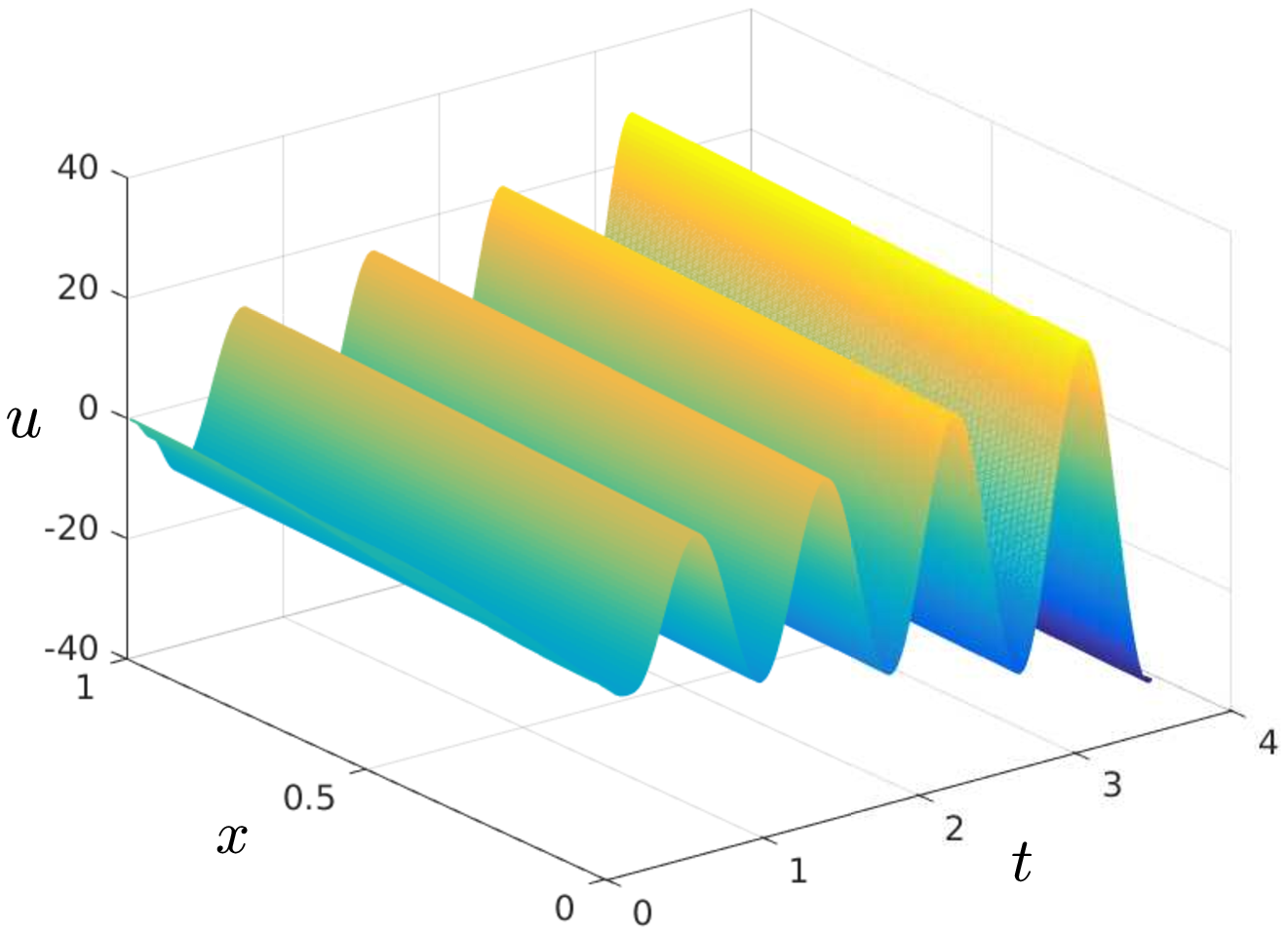}\label{fig:3sim3}}
	\vspace{-0.25cm}
	\caption{Chart of $u$ for system \eqref{eq:simu2} with the parameters: $u^0(x) = \left( \cos(\pi x)+1 \right) \frac{K X_0}{2}$, $X_0 = \left[ 1 \ 1 \right]^{\top}$, $v^0(x) = 0$ and $c_0 = 0.15$ for 3 values of $c$. These results are obtained using Euler forward as a numerical scheme.}
	\label{fig:3sim}
\end{figure*}

\vspace{-0.3cm}
\subsection{Problem \eqref{eq:problem} with $A$ and $A+BK$ not Hurwitz.}

Consider an open loop unstable system defined by:
\vspace{-0.1cm}
\begin{equation} \label{eq:simu3}
	\begin{array}{ccc}
		A = \left[ \begin{smallmatrix} 0 & 1 \\ -2 & 0.1 \end{smallmatrix} \right], &
		B = \left[ \begin{smallmatrix} 0 \\ 1 \end{smallmatrix} \right], &
		K = \left[ \begin{smallmatrix} 1 & 0 \end{smallmatrix} \right].
	\end{array}
\end{equation}

Gain $K$ has been chosen such that the closed loop is also unstable. Surprisingly, the proposed methodology shows the stability for some pairs $(c, c_0)$. The results are presented in Figure \ref{fig:simu3}. The LMIs are not feasible for Theorem 2 with $N = 0$. For $N \supeq 1$, there is a stability area and the slope of the right asymptotic branch is decreasing at each order. Hence, it appears that the introduction of the string equation in the feedback loop helps the stabilization of the closed loop system. For $N = 1$, the stability area is quite far from the maximum one but this gap reduces significantly for higher orders.

\vspace{-0.2cm}
\section{Conclusion}
\vspace{-0.1cm}

%

A hierarchy of stability criteria has been provided for the stability of systems described by the interconnection between a finite dimensional linear system and an infinite dimensional system modeled by a string equation. The proposed methodology relies on an extensive use of Bessel's inequality, which allows to design new and accurate Lyapunov functionals. This new methodology encompasses the classical notion of energy proposed in that case. In particular, the stability of the closed-loop or open-loop system is not a requirement anymore. Future works will include the study of robustness of this approach and the design of a controller.

\vspace{-0.5cm}
\appendix
\subsection{Proof of Lemma \ref{lem:Chi_k}}
\label{sec:Xkprim}

For a given integer $k$ in $\mathbb N$, differentiating of $\Chi_k$ along the trajectories of \eqref{eq:sys2_2}  yields
$\dot \Chi_k 
=c \int_0^1 \chi_x(x) \mathcal{L}_k(x) dx$.
Then, integrating by parts, we get
\begin{equation}\label{diff:Chi_k}
	\begin{array}{lcl}
		\dot\Chi_k & =& c \left( \left[ \chi(x) \mathcal{L}_k(x) \right]_0^1 - \int_0^1 \chi(x) \mathcal{L}'_k(x) dx \right).
	\end{array}
\end{equation}
In order to derive the expression of $\dot\Chi_k $, we use the following properties of the Legendre polynomials. On the one hand, the values of Legendre polynomials at the boundaries of $[0\ 1]$  are given by $\mathcal{L}_k(0) = (-1)^k$ and $\mathcal{L}_k(1) = 1$. On the other hand, the Legendre polynomials verifies the following differentiation rule for $k > 0$:
\begin{equation*}
\frac{d}{dx} \mathcal{L}_k(x) = \sum_{j=0}^{k-1}  (2j\!+\!1)(1\! -\! (-1)^{j+k}) \mathcal{L}_j(x).
\end{equation*}
Hence, injecting these expressions into \eqref{diff:Chi_k} leads to:
\begin{equation*}
	\dot \Chi_k  = c \left( \chi(1,t) - (-1)^k \chi(0) \right) - c \textstyle \sum_{j=0}^{N} \ell_{k,j}\Chi_j,
\end{equation*}
where the coefficient $\ell_{k,j}$ are defined in equation \eqref{def_LN}. The end of the proof consists in gathering the previous expression from $k=1$ to $k=N$, leading to the definition of matrices $L_N$, $\mathbb {1}_N$ and $\bar{\mathbb {1}}_N$ given in \eqref{def_LN}. 

\vspace{-0.1cm}

\bibliographystyle{plain}
\bibliography{report_draft}

\end{document}